\documentclass[graybox]{svmult}

\usepackage{mathptmx}       
\usepackage{helvet}         
\usepackage{courier}        
\usepackage{type1cm}        
\usepackage{makeidx}         
\usepackage{graphicx}        
\usepackage{multicol}        
\usepackage[bottom]{footmisc}

\usepackage{dsfont}
\usepackage{amsmath}
\usepackage{mathrsfs}

\makeindex             

\def\R{\mathds R}

\def\N{\mathds N}
\def\1{\mathds 1}
\def\P{\mathds P}
\def\E{\mathds E}
\def\X{\mathds X}
\def\var{{\rm Var}}
\def\N{\mathscr N}

\def\bfx{\mathbf{x}}
\def\bfy{\mathbf{y}}
\def\bft{\mathbf{t}}
\def\bfm{\mathbf{m}}

\def\exp#1{\,{\rm exp} \left(#1\right)}

\newcommand{\MR}[1]{}
\newcommand{\RL}[1]{}

\begin{document}

\title*{U-statistics in stochastic geometry}
\author{Rapha\"el Lachi\`eze-Rey and Matthias Reitzner}
\institute{Name of First Author \at Name, Address of Institute, \email{name@email.address}
\and Name of Second Author \at Name, Address of Institute \email{name@email.address}}
%
%
\maketitle

\abstract*{Each chapter should be preceded by an abstract (10--15 lines long) that summarizes the content. The abstract will appear \textit{online} at \url{www.SpringerLink.com} and be available with unrestricted access. This allows unregistered users to read the abstract as a teaser for the complete chapter. As a general rule the abstracts will not appear in the printed version of your book unless it is the style of your particular book or that of the series to which your book belongs.
Please use the 'starred' version of the new Springer \texttt{abstract} command for typesetting the text of the online abstracts (cf. source file of this chapter template \texttt{abstract}) and include them with the source files of your manuscript. Use the plain \texttt{abstract} command if the abstract is also to appear in the printed version of the book.}

\abstract{ 
\noindent A U-statistic of order $k$ with kernel $f:\X^k \to \R^d$ over a Poisson process  is defined in \cite{ReiSch11} as
$$ \sum_{x_1, \dots , x_k \in \eta^k_{\neq}} f(x_1, \dots, x_k)  $$ under appropriate integrability assumptions on $f$.
U-statistics play an important role in stochastic geometry since many interesting functionals can be written 
as U-statistics, like intrinsic volumes of intersection processes, characteristics of random  geometric graphs, volumes of random simplices, and many others, see for instance \cite{ LacPec13, LPST,ReiSch11}.
It turns out that the Wiener-Ito chaos expansion of a U-statistic is finite and thus Malliavin calculus is a particularly suitable method. Variance estimates, the approximation of the covariance structure and limit theorems which have been out of reach for many years can be derived.
In this chapter we state the fundamental properties of U-statistics and investigate moment formulae. The main object of the chapter is to introduce the available limit theorems. }


\section{U-statistics and decompositions}
\label{sec:intro}

\subsection{Definition}
Let $\X$ be a Polish space, $k\geq 1$, and $f:\X^{k}\to\mathds{R}$ be a measurable function.
The  $U$-statistic of order $k$ with kernel $f$ over a  configuration   $\eta \in \mathbf{N}_{s}(\X)$   is $0$ if $\eta $ has   strictly less than $k$ points  and the formal sum
\begin{eqnarray*}
U(f;\eta )=\sum_{ \bfx_{k}\in \eta _{\neq }^{k}}f(\bfx_{k})
\end{eqnarray*}
otherwise, where $\eta _{\neq }^{k }$ is the class of $k$-tuples $\bfx_{k} = (x_1, \dots, x_k)$ of distinct points from $\eta $. Remark that since the sum is over all such $k$-tuples, $f$ can be assumed to be symmetric without loss of generality.

An abundant literature deals with the asymptotic study of $U(f;\tilde \eta _{p})$ as $p\to \infty $ when $\tilde \eta _{p}$ is a binomial process, i.e. a set of $p$ iid variables over $\X$. We are  concerned here  with Poisson input, i.e. $\eta$ is a Poisson measure over $\X$ which intensity is a non-atomic locally finite measure $\mu$ on $\X$. So that the definition makes any sense, the basic assumption is that $f\in L_{s}^{1}(\X^{k}) = L_{s}^{1}(\X^{k};\mu^{k})$.\\

In the sequel of this section, let $\mu $ be a non-atomic locally finite measure on $(\X,\mathscr{X}), \eta $ a Poisson measure with intensity $\mu $, and $k\geq 1$.

\subsection{chaotic decomposition and multiple integrals}

\begin{theorem}
\label{th:Ustats-decomp}
Let $f\in L_{s}^{1}(\X^{k})$ such that $U(f;\eta )\in L^{2}(\P)$ . We have the $L^2$ decomposition
\begin{align}
\label{eq:W-I}
U (f;\eta  )=\sum_{n=0}^{k}I_{n}(h_{n}).
\end{align} 
Here $I_{n}$ is the $n$-th order stochastic integral over $\eta  $ defined in Chapter \cite{LastChapter}. The functions $h_{n}$ have been explicitely computed in \cite[Lemma 3.3]{ReiSch11},
\begin{align}
\label{eq:kernel}
h_{n}(\bfx_{n})={\binom n  k} \int_{\X^{n}}f(\bfx_{n},\bfx_{k-n})d\mu ^{k-n}(\bfx_{k-n})
\end{align}
for $\bfx_{n}\in \X^{n}$, and $h_{n}$ is a function of $L^{1}_{s}(\X^{n})\cap L_{s}^{2}(\X^{n})$. \end{theorem}

\begin{remark}
Somewhat counterintuitively, $f\in L_{s}^{1}(\X^{k})\cap L_{s}^{2}(\X^{k})$ does not imply that $\E U(f;\eta )^{2}<\infty $ (see \cite{ReiSch11}), but in most examples   $f$ is bounded and has a bounded support, which makes the latter condition automatically satisfied.
\end{remark}

As is apparent in Theorem \ref{th:Ustats-decomp}, each $U$-statistic of order $k$ is a finite sum of multiple integrals of order $n\leq k$, and it is not difficult to prove that conversely any multiple integral of order $n\geq 1$ can be written as a finite sum of $U$-statistics which orders are smaller or equal to $ n$. From a formal point of view, it is therefore equivalent to study the asymptotics of finite sums of $U$-statistics or of finite sums of multiple integrals.   $U$-statistics are  more likely to appear in applications, but the homogeneity of multiple integrals make them easier to deal with, and some of the Malliavin operators of U-statistics have a particularly intuitive form.
  Consider for instance the case where $F=I_{k}(f)$ is a multiple integral of order $k$. The Malliavin derivative, the Orsntein-Uhlenbeck, and inverse Orstein-Uhlenbeck operators,  take the following form  
\begin{align}
\label{eq:malliavin-integral}
D_{x}F=kI_{k-1}(f(x,\cdot )), x\in \X,\quad  LF=-kI_{k}(f),\quad  L^{-1}F=-k^{-1}I_{k}(f).
\end{align}    
For a $U$-statistic $F$, one can still derive $D_{x}F, LF, L^{-1}F$ using the linearity of those operators and the decomposition (\ref{th:Ustats-decomp}).

The object of this section is really the study of sums of multiple integrals which order is bounded by some $k\geq 1$. The chaotic decomposition also yields that any $L^{2}$ variable can be approximated by such a sum, allowing us in some cases to pass on limit theorems stated here to infinite sums. The following result gives the first two moments of  $U$-statistics. 
\begin{proposition}Let $f\in L_s^{1}(\X^{k})$. Then $\E  | U(f;\eta ) | <\infty $ and 
\begin{align*}
\E U(f;\eta )=\int_{\X^{k}}f(\bfx_{k})\ d\mu^{k}(\bfx_{k}).
\end{align*}
If furthermore $U(f;\eta )\in L^{2}(\P)$, 
\begin{align}
\label{eq:var}
\var(U(f;\eta ))=\sum_{n=1}^{k}n!\|h_{n}\|^{2}
\end{align}
where $h_{n}$  is given in Theorem \ref{th:Ustats-decomp} and $\|h_{n}\|$ is the usual $L^{2}(\X^{n})$-norm.
\end{proposition}
\begin{proof}The first statement is a direct consequence of the Slyvniack-Mecke formula, while the second stems from the orthogonality between multiple integrals $I_{n}(h_{n}), 0\leq n\leq k$.
\end{proof}

\subsection{Hoeffding decomposition}
  
 Let $N\geq 1, \tilde \eta _{p}=\{X_{1},\dots ,X_{p}\}$ be a family of  i.i.d. variables with common distribution $\mu $ on $\X$. Given a measurable kernel $h$ over $\X^{k},k\geq 1$, the traditional Hoeffding decomposition (see e.g. Vitale \cite{Vit92}) is written  
 \begin{align*}
U(h,\tilde \eta _{p})=k!{\binom p k}\sigma _{k}^{N}(h) =k!{\binom p k}\sum_{m=0}^{k}{\binom k m}\sigma ^{p}_{m}(H_{m} ),
\end{align*}
where
\begin{align*}
\sigma _{m}^{p}(H_{m} )=\frac{1}{{\binom N  m}}\sum_{1\leq i_{1}<i_{2}<\dots <i_{m}\leq p}H_{m} (X_{i_{1}},\dots ,X_{i_{m}}),\quad 0\leq m\leq k,
\end{align*} 
and each kernel $H_{m} $ is symmetric and \emph{completely degenerated}, i.e. 
\begin{align*}
\E H_{m}(x_{1},\dots ,x_{m-1},X_{m})=\int_{\X}H_{m}(x_{1},\dots,x_{m-1} ,y)d\mu (y)=0
\end{align*}
for $\mu ^{(m-1)}$-a.e. $x_{1},\dots ,x_{m-1}$.
This property implies in particular the orthogonality of the $\sigma ^{p}_{m}(H_{m}), 1\leq n\leq k$. If $\mu $ is a probability measure, the $H_{m}$ are uniquely defined and can be expressed explicitely via an inclusion-exclusion formula
\begin{align}
\label{eq:WI-H}
H_{m} (x_{1},\dots ,x_{m})=\sum_{n=0}^{m}(-1)^{m-n} \sum_{1\leq i_{1}<\dots <i_{n}\leq m }{\binom k  n}^{-1}h_{n}(x_{i_{1}},\dots ,x_{i_{n}})
\end{align}
where $h_{n}$ is defined in (\ref{eq:kernel}).
As is clear in this last formula, this decomposition is  different from (\ref{eq:W-I}) because in the latter multiple integration is performed with respect to the compensated measure $\eta -\mu $, while in $\sigma _{m}^{p}(H_{m})$ the compensation occurs in the kernel $H_{m}$. 
 
The Hoeffding rank $m_{1}$ is defined as the smallest index $m$ such that $\|H_{m}\|\neq 0$, and we can see through (\ref{eq:WI-H}) that it is equal to the smallest index $n$ such that $\|h_{n}\|\neq 0$. We furthermore have $H_{m_{1}}={\binom k {m_1}}^{-1}h_{m_{1}}$. As proved in \cite{DynMan83} for binomial processes or \cite{LacPec13b} for Poisson processes, the stochastic integral of order $m_{1}$ dominates the sum, and limit theorems for geometric $U$-statistics can then be derived by studying this term, see Section \ref{sec:geometric}.

\subsection{Contraction operators} 

\newcommand{\bfz}{\mathbf{z}}
 
Let $f\in L_{s}^{1}(\X^{q}),g\in L_{s}^{1}(\X^{k})$. If $f$ and $g$ satisfy the technical conditions defined in Chapter \cite{LastChapter}, one can define for $1\leq r\leq l\leq \min(q,k)$ their contraction function of index $(r,l)$, denoted $f\star_{l}^{r}g$. It has $k+q-r-l$ variables as arguments, decomposed in $(\bfx_{r-l},\bfy_{q-r},\bfz_{k-r})$, where $\bfx_{r-l}\in \X^{r-l},\bfy_{q-r}\in \X^{q-r}$, and $\bfz_{k-r}\in \X^{k-r}$. We have
\begin{align*}
f\star_{l}^{r}g(\bfx_{r-l},\bfy_{q-r},\bfz_{k-r}):=
\int_{}f(\bfx_{l},\bfx_{r-l},\bfy_{q-r})g(\bfx_{l},\bfx_{r-l},\bfz_{k-r})\, d\mu^l(\bfx_{l}) .
\end{align*}
Remember that each function appearing here is symmetric, whence the order of the arguments does not matter. Contraction operators are used below to assess the distance between a stochastic integral and the normal law. See \cite{BouPecSurvey} for more information on contraction operators.

 
\section{Rates of convergence}

Let  $F$ be a $L^{2}$ variable of the form 
\begin{align}
\label{eq:general-RV}
F=\sum_{n=0}^{k}I_{n}(h_{n})
\end{align}
for some kernels $h_{n}\in L^{2}_{s}(\X^{n}), n\geq 1$. We assume that those kernels satisfy the technical conditions mentioned in Chapter \cite{BouPecSurvey} so that their mutual contraction kernels are well defined. This model englobes $U$-statistics, as outlined by Theorem \ref{th:Ustats-decomp}, as well as finite sums of $U$-statistics and multiple integrals.  
 
In applied situations, the set-up consists of a fixed integer $k\geq 1$, and, for $t >0$, a family of measures $\mu _{t }$ on $\X$, and a family of kernels $h_{n,t }\in  L_s^{2}(\X^n; \mu _{t }^{n}), 1\leq n \leq k$.  We study the random variables
\begin{align}
\label{eq:Ft}
F_{t}:=\sum_{n=1}I_{n}(h_{n,t}),
\end{align} 
and more precisely the existence of numbers $a_{t },b_{t }>0$ and of a random variable $V$ such that 
\begin{align*}
\tilde F_{t }:=\frac{F_{t }-a_{t }}{\sqrt{b_{t }}}\to V
\end{align*}
in the weak topology. 
In all the applications, $\mu _{t }$ is either of the form 
\begin{itemize}
 \item $\mu _{t }=t \mu $ for some reference measure $\mu $ on the space $\X$, or 
 \item $\mu _{t }=1_{\X_{t }}\mu $ where $\X_{t } \subset \X$ depends on $t $. 
\end{itemize}
The following two settings occur in the most important applications. 

If $\eta= \eta_t$ is a Poisson point process on $\X = \R^d$ the measure $\mu $ will often be the Lebesgue measure $\ell_{d}$, or for $\X=\R^d\times M$ a product measure $\mu =\ell_{d}\otimes \nu $ with a probability measure $\nu $ on a topological marks space $(M,\mathscr{M})$.  

If $\eta= \eta_t$ is a Poisson \rq flat\lq\ process on the Grassmannian $\X=\mathcal{A}_{i}^{d}$ of affine $i$-dimensional subspaces (flats) of $\R^d$, the intensity measure $\mu(\cdot)$ will be a translation invariant measure on $\mathcal{A}_{i}^{d}$.
The Poisson flat process is only observed in a compact convex window $W\subset\R^d$ with interior points. Thus, we can view $\eta_t$ as a Poisson process on the set $[W]$ defined by 
$$[W]=\left\{h\in \mathcal{A}_{i}^{d}:\ h\cap W\neq\emptyset\right\} .$$

\subsection{Central Limit theorem}  \label{sec:CLT}
  
Let $F$ be of the form (\ref{eq:general-RV}). Let $N\sim \N(0,1)$, $\sigma ^{2}=\var(F)$. The next result, which Wasserstein bound has been established in  \cite{LacPec13}, and Kolmogorov bound in \cite{EicTha14},  gives a bound on the distance between $F$ and $N$ in terms of the contractions between the kernels of $F$.
\begin{theorem}\label{th:clt}
Put
\begin{align*}
B(F)&=\max\left( \max_{1}\|h_{n}\star_{r}^{l} h_{n}\|,\max_{2}\|h_{n}\star_{r}^{l} h_{m}\| ,\max_{n=1,\dots ,k}\|h_{n}\|^{2}  \right)\\
B'(F)&=\max( | 1-\sigma ^{2} |,B(F),B(F)^{3/2} ) 
\end{align*}
where
$\max_{1}$ is over $1\leq r\leq n\leq k,1\leq l\leq r\wedge (n-1)$, 
$\max_{2}$ is over $1\leq l\leq r\leq n<m\leq k$. There exists a constant $C_{k}>0$ not depending on the kernels of  $F$ such that
\begin{align}
\label{eq:bound-contractions-wass}d_{W}(F,N)&\leq \sigma ^{-1}C_{k}B(F)\\
\label{eq:bound-contractions-kolmo}d_{K}(F,N)&\leq C_{k}B'(F).
\end{align}

\end{theorem}
  
We reproduce here the important steps of the proof for the Wasserstein bound. The main result, due to Peccati, Sole, Taqqu, Utzet \cite{PSTU10}, is a general inequality on the Wasserstein distance between a Poisson functional $F$ with variance $\sigma ^{2}>0$ having a finite Wiener-Ito expansion and the normal law. We have 
\begin{eqnarray}\label{eq:pstu}
d_{W}(F,N) &\leq& 
\frac{1}{\sigma }\sqrt{\E [(\sigma ^{2}-\langle D_{x}F,-D_{x}L^{-1}F \rangle_{L^{2}(\X )})^{2}]}
 \\&& + \nonumber
\frac{1}{\sigma ^{2}}\int_{\X}\E [(D_{x}F)^{2} |  D_{x}L^{-1}F| ]\mu (dx).
\end{eqnarray}
To translate those inequalities into bounds on the contraction norms, we use the multiplication formula from \cite{PSTU10}, which yields that the multiplication of mutiple integrals is a linear combinations of multiple integrals. For $k,q\geq 1, f\in L_{s}^{2}(\X^{q}),g\in L_{s}^{2}(\X^{k})$,
\begin{align}
\label{eq:multiplication}
I_{q}(f)I_{k}(g)=\sum_{r=0}^{q\vee k}r!{\binom q r}{\binom k  r}\sum_{l=0}^{r}{\binom r  l}I_{q+p-r-l}(f \tilde \star_{r}^{l}g),
\end{align}
where the symmetrized contraction kernels $f\tilde\star_{r}^{l} g$ are the average of kernels $f\star_{r}^{l} g$ over all possible permutations of the variables.

If for instance $F=I_{k}(f)$ is a single multiple integral, (\ref{eq:malliavin-integral}) gives the value of the Malliavin operators, and   a computation then yields the bound \eqref{eq:bound-contractions-wass}  with $f_{k}=f; f_{i}=0$ for $i\neq k$, see \cite[Prop. 5.5]{PecZhe10}.
If $F$ is a general functional with a finite decomposition, such as a $U$-statistic (see (\ref{th:Ustats-decomp})), Malliavin operators are computed using linearity and yield   the bound (\ref{eq:bound-contractions-wass}), see the proof of Th. 3.5 in \cite{LacPec13}.

Concerning Kolmogorov distance, Schulte \cite{Sch12, Sch14} has derived a Stein bound 
similar to (\ref{eq:pstu}), but with more terms on the right-hand side (Theorem 1.1), reflecting the effect that test functions are indicator functions, more irregular than the Lipschitz functions involved in Wasserstein distance. This bound was later improved by Eichelsbacher and Th\"ale \cite[Th. 3.1]{EicTha14}, reducing the number of additional terms.  With similar computations as in the Wasserstein case, one can then prove \cite[Th. 4.1]{EicTha14} that those additional terms only add contraction norms $\|f_{i}\star_{l}^{r}f_{j}\|^{3/2}$ at the power $3/2$, up to a constant, yielding the bound $B'(F)$.

\begin{remark}
The terms in $B'(F)$ bounding the Kolmogorov distance are  smaller than the original terms present in $B(F)$ if the bound goes to $0$, and don't change the bound magnitude or its eventual convergence to $0$.
\end{remark}
\begin{remark}
The constant $C_{k}$ explodes as $k\to \infty $. In other papers \cite{ReiSch11}, \cite{LPST},  similar bounds are derived in more specific cases, with a different method. The constants are more tractable and allow for instance to approximate accurately the distance from a Gaussian to an infinite series of multiple integrals by that of its truncation at some order (see for instance \cite{Sch12}).
\end{remark}

\begin{theorem}[4th moment theorem]  Assume furthermore that kernels $h_{k}$ are non-negative. Then for some $C_{k}'>0$
\begin{align*}
B(F)&\leq  C_{k}'\sqrt{\E F^{4}-3\sigma ^{4}}.
\end{align*}
\end{theorem}
\begin{itemize}\item In view of (\ref{eq:bound-contractions-wass}), the convergence of the $4$-th moment to that of a Gaussian therefore implies central limit, with a bound for Wasserstein distance. In this case, as noted in \cite{EicTha14}, using (\ref{eq:bound-contractions-kolmo}) yields a similar bound for Kolmogorov distance. The positiveness of the kernels is adapted to $U$-statistics with a non-negative kernel.
\item It is highly remarkable that the convergence of the $4$-th moment to that of the Gaussian variable is therefore sufficient for such variables to converge to the normal law. The only technical requirement is that the variables $F_{t}^{4}$ are uniformly integrable.
 
\end{itemize}
 
 \begin{example}[De Jong's theorem]
Let $f_{2}$ be a non-zero degenerate symmetric kernel from $L_{s}^{1}(\X^{2})$, i.e. such that 
\begin{align*}
\int_{\X}f_{2}(x,y)\mu (dx)=0\;\text{ for $\mu $-a.e. $y\in \X$}.
\end{align*} 
This degeneracy property implies that $U(f_{2},\eta )=I(f_{2};\eta )$, we also assume that $f_2\in L_{s}^{4}(\X^{2})$.  De Jong \cite{DeJ90} derived a $4$-th moment central limit theorem for binomial $U$-statistics of the form $U(f_{2};\tilde \eta _{p})$, where $p\in \mathds{N}$ goes to infinity and $\tilde \eta _{p}$ is a sequence of $p$ iid variables with law $\mu $. In the Poisson framework, (\ref{eq:bound-contractions-kolmo}) yields Berry-Essen bounds between $F =U(f_{2};\eta )=I_{2}(f_{2};\eta  )$ and $N$:

\begin{align*}
d_{W}(\tilde F,N)&\leq C_{2} \frac{1}{\|f_{2}\|^{2}}b(f_{2})\\
  d_{K}(\tilde F,N)&\leq C_{2}\frac{1}{\|f_{2}\|^{2}}\max (b(f_{2}),b(f_{2})^{3/2})
\end{align*} 
where 
\begin{align*}
 b(f_{2})=\max\left( \|f_{2}\star_{2}^{0}f_{2}\|,\|f_{2}\star_{1}^{1} f_2\|,\|f_{2}\star _{2}^{1} f_{2}\| \right).
\end{align*}
See Eichelsbacher and Th\"ale \cite[Th. 4.5]{EicTha14} for details. In \cite{PecTha13}, Peccati and Th\"ale derive bounds on the Wasserstein  distance between such a $U$-statistic and a target Gamma variable, also in terms of contraction operators.
\end{example}

\subsubsection{Local marked $U$-statistics}
\label{sec:local}

For many applications, it is useful to assume that the state space $\X$ is of the form $S\times M$ where $S$ is a subset of $\R^d$ containing the points $t_i$ of $\eta $, and $(M,\mathscr{M} )$ is a \emph{mark space}, i.e. a locally compact  space endowed with some probability measure $\nu $. 
The  space $M$ contains marks $m_i$ that will be randomly assigned to each point of the process. 
In this setting, assume that $\eta _{t}$ has intensity measure $\mu _{t}=\1_{\X_{t}}\ell_{d}\otimes \nu $ and let $F_{t} \in L^{2} (\P)$ be a $U$-statistic $F_{t}=U(f ;\eta _{t})$. 
Let the kernel $f$ of the U-statistic be locally square integrable on $\X_{t}=[-t^{1/d},t^{1/d}]^d\times M$ and 
stationary, i.e. for $\mu ^{k}$-almost all $(\bft_{k},\bfm_{k}) \in \X_t^{k},z\in \R^d$,  
\begin{align}
\label{eq:rapidly-decreasing}
f (\bft_{k}+z,\bfm_{k} )=f (\bft_{k},\bfm_{k} ).
\end{align} 
The tail behavior of the function $f$ is fundamental regarding the limit of variables $F_{t}$ as $t\to \infty $.

\begin{definition}A measurable function $f:(\R^d\times M)^{k}\to \mathds{R}$ is rapidly decreasing if it is locally square integrable, stationary, and if it satisfies the following integrability condition: There exists a non-vanishing  probability density $\kappa $ on $(\R^d)^{k-1}$ such that for $p=2,4$,
\begin{align*}
A_{p}(f)=\int_{(\R^d)^{k-1}\times M^{k}}f(0,\bft_{k-1},\bfm_{k})^{p}\kappa(\bft_{k-1})^{1-p}d\ell_{d}^{k-1}(\bft_{k-1})d\nu ^{k}(\bfm_{k})<\infty .
\end{align*}
\end{definition} 
The slight abuse of notation $f(0,\bft_{k-1},\bfm_{k-1})$ means that $\bft_{k} = (0,\bft_{k-1}) = (0,t_{2},\dots ,t_{k-1})$, and $\bfm_{k}=(m_{1},\dots ,m_{k})$.

We have in this case the following result, which is a consequence of Theorem 6.2 and Example 2.12-(ii) in \cite{LacPec13b} :
\begin{theorem}\label{thm:local-Ustats}
Let $F_{t}=U(f;\eta _{t})$ where $f $ is a rapidly decreasing function, and $\mu _{t}=\1_{\X_{t}}\ell_{d} \otimes \nu $ with $\X_{t}=[-t^{1/d},t^{1/d}]^{d} \times M$. Then, with $a_{t}=\E F_{t},b_{t}=\var(F_{t}) $, we have for some $C_{1},C_{2},C_{3}>0$ not depending on $t$,
\begin{align*}
C_{1}t &\leq b_{t} \leq C_{2}t \\
d_{W}(\tilde F_{t},N_{1})&\leq C_{3} t^{-1/2}.
\end{align*}
\end{theorem}

\begin{remark}
Reitzner \& Schulte \cite{ReiSch11} first established this result in the case where $f$  is the indicator function of a ball of $\R^d$ (any non-vanishing continuous density $\kappa $ can be chosen in this case because $f(0,\cdot )$ has a compact support).
\end{remark}
\begin{remark}
A similar result holds if $F$ is simply assumed to be a finite sum of stochastic integrals which kernels are rapidly decreasing functions, the $U$-statistics being a particular case.
\end{remark}

 \subsubsection{Geometric $U$-statistics} 
 \label{sec:geometric}
 Coming back to the general framework, assume $F_{t}=U(f;\mu _{t})$ where $f\in L_{s}^{2}(\X^{k})$ is fixed and  $\mu _{t}=t\mu $ for some measure $\mu $ on $\X$.  Then $F_t$ admits the decomposition (\ref{eq:Ft}) where 
\begin{align*}
h_{n,t}(\bfx_{n})=t^{k-n}{\binom k n}\int_{\X^{k-n}} f(\bfx_{n},\bfx_{k-n})d\mu ^{k-n}(\bfx_{k-n}),\, \bfx_{n}\in \X^{n}.
\end{align*}
One can then see that the term $\|h_{1,t}\|$ dominates the other terms in the variance expression (\ref{eq:var}), provided this term does not vanish. In any case 
 the important feature is the Hoeffding rank of the $U$-statistic 
 \begin{align*}
n_{1}:=\inf \{n:\|h_{n}\|\neq 0\},
\end{align*}
because it turns out that $I_{n_{1}}(h_{n_{1},t})$ is the predominant term in (\ref{eq:Ft}), in the sense that $F_{t}-I_{n_{1}}(h_{n_{1},t})=o(F_{t})$ for the $L^{2}$ norm as $t\to \infty $. It yields the following result (Theorem 7.3 in \cite{LacPec13b}).
\begin{theorem}
\label{thm:geometric-Ustats}
For some $C_{1},C_{2},C_{3}>0$ not depending on $t$,
\begin{align*}
C_{1}t^{2k-n_{1}}\leq b_{t}\leq C_{2}t^{2k-n_{1}}.
\end{align*}
(i) If $n_{1}=1$, $U(f;\mu _{t})$ follows a central limit theorem and 
\begin{align*}
d_{W}(\tilde F_{t},N)\leq C_{3}t^{-1/2},\\
d_{K}(\tilde F_{t},N)\leq C_{3}t^{-1/2}.
\end{align*}
(ii) If $n_{1}>1$, $U(f;\mu _{t})$ does not follow a CLT and $\tilde F_{t}$ converges to a Gaussian chaos of order $n_{1}$ (see \cite[Theorem 7.3-2]{LacPec13b}).
\end{theorem}
For a deeper understanding we refer to the proof of Theorem \ref{thm:regimes}.
\begin{remark}
Point (i) first appears in \cite{ReiSch11}.\end{remark}
\begin{remark}
Point (ii) crucially uses the results of Dynkin \& Mandelbaum \cite{DynMan83}.
\end{remark}
\begin{remark}
The speed of convergence to the Gaussian chaos in (ii) is studied by Peccati and Th\"ale \cite{PecTha13} in case the limit is a Gamma distributed random variable.
\end{remark}

 \subsubsection{Regimes classification}
 
The crucial difference in Theorems \ref{thm:local-Ustats} and \ref{thm:geometric-Ustats} is the area of influence of a given point $x\in \eta _{t}$. In the case of a local $U$-statistic, a typical point $x\in \eta _{t}$ interacts with a stochastically bounded number of neighbors, that are more likely near in view of Assumption \ref{eq:rapidly-decreasing}. The situation is different for a geometric $U$-statistic, where  a point potentially  interacts with any other point, regardless of the distance. 
Both these  regimes can be seen as two particular cases of a continuum. 
 
Let $\alpha _{t}>0$ be a \emph{scaling factor}$, \X_{t}=[-t^{1/d},t^{1/d}]^{d} \times M$, 
$\mu _{t}=\1_{\X_{t}}\ell_{d}\otimes \nu $, and $F_{t}=U(f_{t};\eta _{t})$, where $f_{t}$ is obtained by rescaling  a rapidly decreasing function $f$:
\begin{align}
\label{eq:rescaling}
f_{t}(\bfx_{k})=f(\alpha _{t} \bfx_{k}), \bfx_{k}\in \X_{t}^{k}.
\end{align} 
Say that $f$ has non-degenerate projections if the functions 
\begin{align*}
f_{n}(\bfx_{n})=\int_{\X_t^{k-n}}f(\bfx_{n},\bfx_{k-n})d\mu_t^{k-n}, \bfx_{n}\in \X^{n},
\end{align*}
well defined in virtue of  (\ref{eq:rapidly-decreasing}), are not $\mu $-a.e. equal to $0$. It is trivially the case if for instance $f\neq 0$ and $f\geq 0$ $\mu $-a.e..
Concerning notations, every spatial transformation of a point $x=(t,m)\in \R^d\times M$, such as translation, rotation, or multiplication by a scalar number, is only applied to the spatial component $t$.

Subsequently, any spatial transformation applied to a $k$-tuple of points $\bfx_{k}=(x_{1},\dots ,x_{k})$ is applied to the spatial components of the $x_{i}$'s.
The quantity 
$v_{t }=\alpha _{t }^{-d}$ 
is relevant because it gives the magnitude of the number of points interacting with a typical point $x$. The case $v_{t}=\alpha _{t}=1$ is that of local $U$-statistic. If $v_{t}=t $ is roughly  the volume of $\X_{t}$, it corresponds to geometric $U$-statistics. In this case it is useless to assume that $f$ is rapidly decreasing, as only the behavior over $\X_{1}$ is relevant for the problem. 

\begin{theorem}\label{thm:regimes}
Assume that $f_{t}$ is of the form (\ref{eq:rescaling}), where $f$ is a rapidly decreasing function  with non-degenerate projections. With the notations above, there are $C_{1},C_{2},C_{3}>0$ such that 
\begin{align*}
C_{1}\leq \frac{b_{t}}{t v_{t}^{2k-2}\max(1,v_{t}^{-k+1})}\leq C_{2},
\end{align*}
and 
\begin{align*}
d_{W}(\tilde F_{t},N)&\leq C_{3}
 t^{-1/2}\max(1,v_{t}^{-k+1})^{1/2} \\
 d_{K}(\tilde F_{t},N)&\leq C_{3}
 t^{-1/2}\max(1,v_{t}^{-k+1})^{1/2} .
\end{align*}
\end{theorem}

Concerning the bound for Kolmogorov distance, it is not formally present in the literature. It relies on the fact that in Theorem \ref{th:clt}, $B'(F)\leq C B(F)$ for some $C>0$ in the case where $\sigma \to 1$ and $B(F)\to 0$. Then one can simply reproduce the proof of \cite{LacPec13b}, entirely based on an upper bound for $B(F)$.

\begin{remark}
Theorems \ref{thm:local-Ustats} and \ref{thm:geometric-Ustats}-(i) can be retrieved from this theorem by setting respectively $v_{t}=1$ or $v_{t}=t$.
\end{remark}
\begin{remark}
If some projections do vanish, the convergence rate can be modified, and the limit might not even be gaussian, as it is the the case for  the degenerate geometric $U$-statistics of Th. \ref{thm:geometric-Ustats}-(ii).  
\end{remark}
\begin{remark}
\label{rk:regimes}
Depending on the asymptotic behavior of $v_{t}$, we can identify four different regimes:\begin{enumerate}
\item \textbf{Long interactions}:  $v_{t}\to \infty $, CLT at speed $t^{-1/2}$, the first chaos  $I_{1,t}(h_{t,1})$ dominates (geometric $U$-statistics).
\item \textbf{Constant size interactions}: $v_{t}=1$, CLT at speed $t^{-1/2}$, all chaoses have the same order of magnitude (local $U$-statistics).
\item \textbf{Small interactions}:  $v_{t}\to 0, t v_{t}^{-k+1}\to \infty $, CLT at speed $(t v_{t}^{-k+1})^{-1/2}$, higher order chaoses dominate. In the case of random graphs $(k=2)$, the corresponding bound in $(tv_{t})^{-1/2}$ has been obtained in \cite{LacPec13}.
\item \textbf{Rare interactions}: $tv_{t}^{-k+1}\to c<\infty $, the bound does not converge to $0$. In the case $k=2$, it has been shown in \cite{LacPec13} that there is no CLT but a Poisson limit in the case $c>0$ (see Chapter \cite{BouPecSurvey} for more on Poisson limits).
\end{enumerate}
\end{remark}

 \subsection{Other limits and multi-dimensional convergence}
 
 Besides the Gaussian chaoses appearing in Theorem \ref{thm:geometric-Ustats}-(ii), some characterizations of non-central limits have also been derived for Poisson $U$-statistics. 
 
 \subsubsection{Multidimensional convergence} We consider in this section the conjoint behavior of random variables $F_{t}=(F_{1,t},\dots ,F_{k,t})$ where $F_{m,t}=I_{ q_m}(h_{m,t})$ for $1\leq m\leq k$, with $h_{ {m},t}\in L_s^{2}(\X^{ {q_{m}}})$ for some $q_{m}\geq 1$, for $t>0;1\leq m\leq k$. 

Call $\sigma _{t}^{2}=\sum_{m=1}^{k}\var(F_{m,t})$. 
Any $L^{2}$ candidate for the limit of $\sigma _{t}^{-1}F_{t}$ should have as covariance matrix $$C_{m,n}=\lim_{t} \sigma _{t}^{-2}\E F_{m,t}F_{n,t} , \quad1\leq m,n\leq k$$ if  those limits exist. In this case there is indeed asymptotic normality   if  all contraction norms  
\begin{align*}
\| h_{m,t}\star_{r}^{l}h_{n,t} \|
\end{align*} 
go to $0$ for $r=1,\dots ,q_{k}$, and every $l=1,\dots ,r\wedge (q_{k}-1)$, under technical conditions on the kernels related to technical condition of chapter \cite{BouPecSurvey}, see \cite[Th. 5.8]{PecZhe10},\cite[Th. 2.4]{BouPec12} for details. These articles contain explicit bounds on the speed of convergence with a specific distance related to thrice differentiable functions on $(\R^d)^{k}$, and the convergence is stable, in the sense of \cite{BouPec12}.

 If now $F_{t}=(F_{1,t},\dots ,F_{k,t})$ where each $F_{m,t}$ is a $U$-statistic, one can consider the random vector $G_{t}$ constituted by all multiple integrals with respect to kernels from the decompositions of the $F_{m,t}$, as defined in (\ref{eq:kernel}). One can then infer conditions  for asymptotic normality of $F_{t}$ by applying the previous considerations to $G_{t}$. 
 
 As noted in Remark \ref{rk:regimes}, some $U$-statistics behave asymptotically as Poisson variables. Asymptotic joint laws of $U$-statistics can also converge to random vectors with marginal Poisson laws, and it can also happen that they converge to an hybrid random vector which has both Gaussian and Poisson marginals, here again the reader is referred to Chapter \cite{BouPecSurvey}. 

 \subsubsection{Gamma} 
 Similar results to those of Section \ref{sec:CLT} with Gamma limits have been derived by Peccati and Th\"ale \cite{PecTha13} for Poisson chaoses of even order. The distance used there is 
\begin{align*}
d_{3}(U,V)=\sup_{h\in \mathcal{H}^{3}} |  \E h(U)-h(V)| 
\end{align*} 
where $\mathcal{H}^{3}$ is the class of functions of class $\mathcal{C}^{3}$ with all first $3$ derivatives uniformly bounded by $1$. 
We again denote by $f\tilde\star_{r}^{l} g$ the symmetrized contraction kernels . 

For $\nu >0$, let $F(\nu /2)$ be a Gamma distribution with mean and variance both equal to $\nu /2$. We introduce the centered unit variance variable $G(\nu ):=2F(\nu /2)-\nu $.\begin{theorem}Let $F=I_{k}(h_{k})$ for some even integer $k\geq 2$.

We have 
\begin{align*}
d_{3}(I_{k}(h_{k}),G(\nu ))\leq D_{k}\max\{k!\|h_{k}\|-2\nu ;\|h_{k}\star_{p}^{p}h_{k}\|;\|h_{k}\star_{r}^{l}h_{k}\|^{1/2};\|h_{k}\tilde\star_{q/2}^{q/2}h_{k}-c_{k}h_{k}\|\}
\end{align*}
where the maximum is taken over all $p=1,\dots ,k-1$ such that $p\neq k/2$ and all $(r,l)$ such that $r\neq l$ and $l=0$, or $r\in \{1,\dots ,k\}$ and $l\in \{1,\dots ,\min(r,k-1)\}$. Also 
\begin{align*}
c_{k}=\frac{4}{\left( q/2 \right)!{\binom q {q/2}}^{2}}.
\end{align*}
\end{theorem}

\begin{remark}
In the case of double integrals $(k=2)$, the authors of \cite{PecTha13} provide a 4th moment theorem, in the sense that under some technical conditions, a sequence of double stochastic integrals converge to a Gamma variable if their first moments converge to those of a Gamma variable.
\end{remark}
\begin{remark}
This result enables to give an upper bound on the speed of convergence to the second Gaussian chaos in Theorem \ref{thm:geometric-Ustats} in the case $n_{1}=2$, if this limit is indeed a Gamma variable.
\end{remark}



\section{Large deviations}
There are only few investigations concerning concentration inequalities for Poisson U-statistics. Most results require an  nice bound on 
$ \sup_{\eta \in {\mathbf N}(\X),\ z \in \X} D_z(F) < \infty.$
For U-statistics of order  $\geq 2$ this condition is not satisfied, even if $f$ is bounded. For U-statistics of order $1$, this holds if $\| f \| _\infty < \infty$.
Therefore we split our investigations into a section on U-statistics of order one and on higher order local U-statistics.
We start with a general result. Throughout this section we assume that $f \geq 0$ and $f\neq 0$.

\subsection{A general LDI}

In this section we sketch an approach developed in \cite{ReiSchTh13} leading to a general concentration inequality.
For two counting measures $\eta $ and $ \nu$ we define the difference $\eta  \backslash \nu $ by
\begin{equation} \label{def:xi-nu}
\eta  \backslash \nu  = 
\sum_{x \in \X} (\eta (x)-\nu(x))_+ \, \delta_x\, .
\end{equation}

For $x \in \eta $ and $f\in L_{s}^{1}(\X^{k})$, we recall that
$$ 
U (f; \eta)  = \sum_{x\in\eta } F (x; \eta ) 
\ \ \ \mbox{ with } \ \ \  
F (x; \eta ) =  \sum_{\bfx_{k-1} \in (\eta\backslash\{\bfx\})^{k-1}_{\neq} } f(x,\bfx_{k-1}).
$$ 

\medskip
Assume that in addition to $\eta $ a second point set $\zeta \in {\mathbf N}(\X)$ is chosen.
The non-negativity of $f$ yields 
\begin{eqnarray*}
U(f; \eta )
&\leq &
U(f; \zeta)
+  k \sum_{x \in {\eta }}  F(x; \eta ) \1(x \notin \zeta)
\\ &= &
U(f; \zeta)  +
k \int F(x; \eta ) \, d (\eta  \backslash \zeta) \,.
\end{eqnarray*}

\medskip
The convex distance of a finite point set $\eta  \in {\mathbf N}(\X) $ to some $ A \subset {\mathbf N}(\X)$ was introduced in \cite{Rei13}, and is given by
\begin{eqnarray*}
d_T^\pi ({\eta },A)  &=& \max_{ \| u \|_{2, \eta } \leq 1 }\  \min_{ \zeta \in A}
\int u \ d(\eta  \backslash \zeta)
\end{eqnarray*}
where $u: \X \to \R_+$ is a non-negative measurable function and $\| u \|_{2 ,\eta }^2= \int u^2 d\eta $.
To link the convex distance to the U-statistic, we insert for $u$ the normalized function
$  \| F (x; \eta ) \|_{2, \eta } ^{-1}  F (x; \eta ) $
and rewrite $U(\eta )$ in terms of the convex distance as follows:
\begin{eqnarray*}
d_T^\pi ({\eta },A)
& \geq &   \nonumber
\min_{ \zeta \in A}
\int \frac{1}{\| F (x; \eta ) \|_{2, \eta } } F (x; \eta )  d(\eta  \backslash \zeta)
\\ & \geq & \nonumber
 \frac{1}{k \| F (x; \eta ) \|_{2, \eta } }
\min_{ \zeta \in A}\Big( U(f;\eta ) -U(f; \zeta) \Big).
\end{eqnarray*}

If we assume $F(x; \eta ) \leq B$, then 
$ \|F (x; \eta ) \|_{2, \eta }^2  \leq B  \sum _{x\in \eta }F (x; \eta ) = B U(f;\eta) $,
which implies
\begin{equation} \label{eq:dTpi}
d_T^\pi ({\eta },A)   \geq 
\frac{1}{k \sqrt { B } }
\min_{ \zeta \in A}  \frac{U(f;\eta ) - U(f;\zeta)}{ \sqrt{U(f;{\eta })}} \ \1(\forall x \in \eta:\ F(x; \eta ) \leq B) \, .
\end{equation}

\bigskip
In \cite{Rei13}, a LDI for the convex distance was proved. For  $\eta$ a Poisson point process, and for $A \subset {\bf N}(\X) $, we have
$$ \P (A) \P  \left( d_T^\pi({\eta }, A) \geq s  \right) \leq  \exp{- \frac {s^2} 4}  .  $$
Precisely as in \cite{ReiSchTh13}, this concentration inequality combined with the estimate (\ref{eq:dTpi}) yields the following theorem.
\begin{theorem}\label{th:LDIgen}
Assume that $\varepsilon(\cdot)$ and $B\in \R$ satisfy
$ \P( \exists x \in \eta :\ F(x; \eta ) >B) \leq \varepsilon (B)$.  
Let $m$ be the median of $U(f; \eta)$.
Then
 \begin{equation}
\P( | U(f; \eta ) -m |\geq u)
\leq
4 \exp{- \frac {u^2 }{ 4 k^2 B  (u+m)}}
+  3 \varepsilon (B) \,.
\end{equation}
\end{theorem}
In the next sections we apply this to U-statistics of order one and to local U-statistics. In the applications, the crucial ingredient is a good estimate for $\varepsilon (B)$.

\subsection{LDI for first order U-statistics }

There are several concentration inequalities for integrals over Poisson point processes, i.e. U-statistics of order one,
$$U(f;\eta ) = \sum_{x \in \eta} f(x) = \int f d\eta, f \geq 0$$
in which case $ D_z U = f(z)$. Assuming that $\| f\|_\infty =B < \infty$ we have 
$$ \| D_z U \| _\infty  \leq B . $$
A result by Houdre and Privault \cite{HouPri} shows that 
\begin{equation}\label{eq:LDIHP}
  \P( U - \|f\|_1 \geq u ) \leq \exp{ - \frac {\|f\|_1}{\| f\|_\infty} g\big( \frac u{\|f\|_1} \big)}
\end{equation}
where $g(u)= (1+u) \ln (1+u) -u,\ u \geq 0$ and because $f \geq 0$ the 1-norm equals the expectation $\E U(f; \eta)$.
A similar result is due to Ane and Ledoux \cite{AneLed}. 
Reynaud-Bouret \cite{ReBo} proves an estimate involving the 2-norm $\| f\|_2$ instead of the 1-norm. 
A slightly more general estimate is given by Breton et al. \cite{BHP}.

We could also make use of Theorem \ref{th:LDIgen} and choose $B= \| f \|_\infty$. This yields
 \begin{equation}
\P( | U(f; \eta ) -m |\geq u)
\leq
4 \exp{- \frac {u^2 }{ 4 \| f\|_\infty  (u+m)}} \,,
\end{equation}
which is a slightly weaker estimate than (\ref{eq:LDIHP}).

\subsection{LDI for local U-statistics}
In this paragraph we assume that $\X$ is equipped with a distance and $B(x,r)$ denotes the ball of radius r around $x \in \X$. 
If $U$ is  a local U-statistic which is concentrated on a ball of radius $\delta_t$, we have 
$$ F(x, \eta) \leq \| f\|_\infty \eta(B(x,\delta_t))^{k-1}$$
\begin{eqnarray*}
\P(\exists x:\  F(x; \eta ) >B) 
&\leq &
\E \sum_{x\in \eta } \1 (F(x; \eta) > B)
\\ &\leq &
 \int_{\X} \P (F(x; \eta) > B) d\mu_t
\end{eqnarray*}
and it remains to estimate
\begin{eqnarray*}
\P \left(\eta(B(x,\delta_t)) > \left( \frac{B}{\|f\|_\infty}\right)^{\frac 1 {k-1}} \right).
\end{eqnarray*}
We use the Chernoff bound for the Poisson distribution, namely
\begin{equation} \label{eq:chern}
\P(\eta_t(B^d(x, \delta_t)) > r) \leq \inf_{s\geq 0} e^{E(e^s-1)-sr} ,
\end{equation}
because  $\eta(B(x,\delta_t))$ is a Poisson distributed random variable with mean
\begin{equation}\label{eq:boundnupi}
  E(x) := \E \eta_t(B^d(x, \delta_t))  = \mu_t(B^d(x, \delta_t))   \leq \sup_{x \in \X}\mu_t(B^d(x, \delta_t)) =: E\,.
\end{equation}
Because
$\inf_{s\geq 0} {E(e^s-1)-sr}
= r(1-\ln(r/E)) -E $ 
we estimate the right hand side of (\ref{eq:chern}) by 
$\exp{- \frac 12 r  } $ for $ E e^2 \leq  r $.
This leads to
\begin{eqnarray*}
\P(\exists x:\  F(x; \eta ) >B) 
&\leq &
\mu_t(\X) \exp{- \frac 1{2} \left(\frac{B}{\|f\|_\infty}\right)^{\frac 1 {k-1}} } 
:= \varepsilon (B)
\end{eqnarray*}
for $ B  \geq E^{k-1} e^{2(k-1)} \|f\|_\infty$.
We set $ B = \|f\|_\infty^{\frac 1k} (\frac {u^{2} }{   (u+m)})^{\frac{k-1}k}  $ 
and combine this with the general Theorem \ref{th:LDIgen}.
\begin{theorem}\label{th:localLDI}
Set $E:= \sup_{x \in \X}\mu_t(B^d(x, \delta_t))  $. 
Then for $ \frac {u^{2} }{   (u+m)} \geq E^{k} e^{2k} \|f\|_\infty $,
\begin{eqnarray*}
\P( | U(f; \eta ) -m |\geq u) 
\leq 4 \mu_t(\X) \exp{- \frac{1}{4 k^2} \|f\|_\infty^{-\frac 1k}  \Big(\frac {u^{2} }{  u+m}\Big)^{\frac{1}k} } .
\end{eqnarray*}
\end{theorem}
Clearly, in particular situations more careful choices of $\varepsilon (B)$ and $B$ lead to more precise bounds.


\section{Applications}
 \label{sec:appli}
In this section we investigate some applications  of the previous theorems in stochastic geometry. In all these cases $\X$ is either a subset of $\R^d$ or a subset of the affine Grassmannian ${\mathcal A}_i^d$, the space of all $i$-dimensional spaces in $\R^d$.

We state some normal approximation and concentration results which follow from the previous theorems.
In many cases multi-dimensional convergence and convergence to other limit distributions can be proved in various regimes. We restrict our presentation to certain \lq simple\rq\ cases without making any attemp for completeness. Our aim is just to indicate recent trends, we refer to further results and investigations in the literature.

\subsection{Intersection process}
Let $\eta_t$ be a Poisson process on the space $\mathcal{A}_{i}^{d}$ with an intensity measure of the form 
$\mu(\cdot)=t \theta(\cdot)$ 
with $t \in\R^+$ and a $\sigma$-finite non-atomic measure $\theta$. The Poisson flat process is only observed in a compact convex window $W\subset\R^d$ with interior points. Thus, we can view $\eta_t$ as a Poisson process on the set $\X=[W]$ defined by 
$$[W]=\left\{h\in A_{i}^{d}:\ h\cap W\neq\emptyset\right\} .$$

Given the hyperplane process $\eta_t$, we investigate the $(d-k(d-i))$-flats in $W$ which occur as the intersection of $k$ planes of $\eta_t$.
Hence we assume $k \leq d/(d-i)$.
In particular, we are interested in the sum of their $j$-th intrinsic volumes given by
$$ \label{eqn:Intersections} 
\Phi_t= \Phi_t(W,i,k,j)=
\frac{1}{k!} \sum_{(h_1,\hdots,h_k)\in\eta^{k}_{\neq}} 
V_j(  h_1\cap\hdots\cap h_k\cap W )
$$
for $j=0,\hdots,d-k(d-i)$, $i=0, \dots, d-1$ and $k=1,\hdots,\lfloor d/(d-i) \rfloor$. 
For the definition of the $j$-th intrinsic volume $V_j(\cdot)$ we refer to the Chapter 2 of the current book. We remark that $V_0(K)$ is the Euler characteristic of the set $K$, and that $V_{n} (K) $ of an $n$-dimensional convex set $K$ is the Lebesgue measure $\ell_n (K)$. Thus $\Phi_t(W,i,1,0)$ is the number of flats in $W$ and $\Phi_t(W,i,k,d-k(d-i))$ is the $(d-k(d-i))$-volume of their intersection process.
To ensure that the expectations of these random variables are neither 0 nor infinite, we assume that $0<\theta([W])<\infty$,
and that $2\leq k \leq \lfloor d/(d-i) \rfloor$ independent random hyperplanes on $[W]$ with probability measure $\theta(\cdot)/\theta([W])$ intersect in a $(d-k(d-i))$-flat almost surely and their intersection flat hits the interior of $W$ with positive probability.
For example, these conditions are satisfied if the hyperplane process is stationary and the directional distribution is not concentrated on a great subsphere. 

The fact that the summands in the definition of $\Phi_i^k$ are bounded and have a bounded support makes sure that all moment conditions are satisfied and we can apply Theorem \ref{thm:geometric-Ustats}:
\begin{theorem}
Let $N$ be a standard Gaussian random variable. Then constants $c=c(W,i,k,j)$ exist such that
\begin{align*}
d_{W}(\tilde \Phi_{t},N)\leq c t^{-1/2},\\
d_{K}(\tilde \Phi_{t},N)\leq c t^{-1/2},
\end{align*}
for $t\geq1$.
\end{theorem}

Furthermore, it can be shown \cite{ReiSch11} that the asymptotic variances satisfies \linebreak
$\var \Phi_t= C_{\Phi}  t^{2k-1}(1+o(1))$ as $t \to \infty$ with a constant 
$ C_{\Phi}=C_{\Phi}(W,i,k,j)   $.
The order of magnitude already follows from the first part of Theorem~\ref{thm:geometric-Ustats}.

For more information we refer to \cite{HeinrichSchmidtSchmidt2006} and \cite{LPST}. In the second paper the Wiener-It\^o chaos expansion is used to derive even multivariate central limit theorems in an increasing window for much more general functionals $\Phi$.

\subsection{Flat processes}
 
For $i < \frac d2$ two $i$-dimensional planes in general position will not intersect. Thus the intersection process described in the previous section will be empty with probability one. A natural way to investigate the geometric situation in this setting is to ask for the distances between this $i$-dimensional planes, or more general for the so-called proximity functional. The central limit theorems described in the following fits precisely into the setting of this contribution, we refer to \cite{SchulteThaele2013} for further results.

Let $\eta_t$ be a Poisson process on the space $A(d,i)$ with an intensity measure of the form 
$\mu_t(\cdot)=t \theta(\cdot)$ 
with $t \in\R^+$ and a $\sigma$-finite non-atomic measure $\theta$. The Poisson flat process is observed in a compact convex window $W\subset\R^d$. To two $i$-dimensional planes in general position there is a unique segment $[x_1, x_2]$ with
$$ d(h_1, h_2)= \| x_2-x_1\| = \min_{y \in h_1,z \in h_2} \| z-y\| . $$
The midpoints $m(h_1, h_2)=\frac 12 {(x_1+x_2)}$ form a point process of infinite intensity, hence we restrict this to the point process
$$ \{ m(h_1,h_2):\ d(h_1, h_2) \leq \delta, h_1, h_2 \in \eta_{\neq}^2 \} $$
and are interested in the number of midpoints in $W$.
$$ \Pi_t = \Pi_t(W, \delta) = \frac 12 \sum_{h_1, h_2 \in \eta_{\neq}^2} \1(d(h_1, h_2) \leq \delta, m(h_1, h_2) \in W ) $$
It is not difficult to show that $\E \Pi_t $ is of order $t^2 \delta^{d-2i}$. The U-statistic $\Pi_t$  is local on the space $\mathcal{A}_{i}^{d}$. Thus the following theorem due to Schulte and Thaele \cite{SchulteThaele2013} is in spirit similar to Theorem \ref{thm:local-Ustats}.
\begin{theorem}
Let $N$ be a standard Gaussian random variable. Then constants $c(W,i)$ exist such that
\begin{align*}
d_{K}(\tilde \Pi_{t},N)\leq c (W,i) t^{-\frac{d-i}2}.
\end{align*}
for $t\geq1$.
\end{theorem}
Moreover, Schulte and Th\"ale proved that the ordered distances form after suitable rescaling asymptotically 
an inhomogeneous Poisson point process on the positive real axis.

We add to this a concentration inequality which follows immediately from Theorem~\ref{th:localLDI}.
Observe that $\mu_t(\X) = t \theta( [W])$.
\begin{theorem}
Denote by $m$ the median of $\Pi_t$. 
Then
\begin{eqnarray*}
\P( | \Pi_t -m |\geq u) 
\leq 4 t \theta( [W]) \exp{- \frac{1}{16} \frac {u}{   \sqrt{u+m}}}
\end{eqnarray*}
for 
$ \frac {u}{   \sqrt{u+m}} \geq e^{2} t \sup_{h \in [W]}\theta(B^d(h, \delta))  $.
\end{theorem}

\subsection{Gilbert graph}
Let $\eta_t$ be a Poisson point process on $\R^d$ with an intensity-measure of the form $\mu_t(\cdot)= t \ell_d(\cdot \cap W)$, where $\ell_d$ is Lebesgue measure and $W\subset\R^d$ a compact convex set with $\ell_d(W)=1$. 
Let $(\delta_t:t>0)$ be a sequence of positive real numbers such that $\delta_t\to 0$, as $t\to\infty$.  The random geometric graph is defined by taking the points of $\eta_t$ as vertices and by connecting two distinct points $x,y\in\eta_t$ by an edge if and only if $ \|x-y\| \leq \delta_t $. The resulting graph is called Gilbert graph.

There is a vast literature on the Gilbert graph and one should have a look at Penrose's seminal  book \cite{Penrose03}. More recent developments are due to Bourguin and Peccati \cite{BouPec12},  Lachi\`eze-Rey and Peccati \cite{LacPec13,LacPec13b} and 
Reitzner, Schulte and Th\"ale \cite{ReiSchTh13}. 

In a first step one is interested in the number of edges 
$$
N_t=N_t(W, \delta_t)=\frac{1}{2}\sum_{(x,y)\in \eta^2_{t,\neq}}\1(\| x-y\| \leq \delta_t) 
$$
of this random geometric graph. It is natural to consider instead of the norm functions $\1(f(y-x) \leq \delta_t)$ and instead of counting more general functions $g( y-x)$:
$$\sum_{(x,y)\in \eta^2_{t,\neq}}\1(f(y-x) \leq \delta_t) \, g(y-x). $$
For simplicity we restrict our investigations in this survey to the number of edges $N_t$ in the thermodynamic setting where $t \delta_t^d$ tends to a constant as $t \to \infty$. Further results for other regimes, multivariate limit theorems and sharper concentration inequalities can be found in Penrose's book and the papers mentioned above.

Because of the local definiton of the Gilbert graph, $N_t$ is a local U-statistic. Theorem \ref{thm:regimes} with $\nu_t = t \delta_t^d$ can be applied. 
\begin{theorem}
Let $N$ be a standard Gaussian random variable. Then constants $c(W)$ exist such that
\begin{align*}
d_{W}(\tilde N_{t},N)\leq c(W ) t^{-1/2}, \\
d_{K}(\tilde N_{t},N)\leq c(W) t^{-1/2} .
\end{align*}
for $t\geq1$.
\end{theorem}

A concentration inequality follows immediately from Theorem~\ref{th:localLDI}.
Observe that $\mu_t(\X) = t \ell_d(W)$.
\begin{theorem}
Denote by $m$ the median of $N_t$. 
Then there is a constant $c_d$ such that 
\begin{eqnarray*}
\P( | \Pi_t -m |\geq u) 
\leq 4 t \ell_d(W) \exp{- \frac{1}{16} \frac {u}{   \sqrt{u+m}}}
\end{eqnarray*}
for 
$ \frac {u}{   \sqrt{u+m}} \geq c_d $.
\end{theorem}

In \cite{ReiSchTh13} a concentration inequality for all $u \geq 0$ is given using a similar but more detailed approach.

\subsection{Random simplicial complexes}
Given the Gilbert graph of a Poisson point process $\eta_t$ we construct the Vietoris-Rips complex $R(\delta_t)$ by calling 
$F=\{x_{i_1},\dots , x_{i_{k+1}}\}$ a $k-$face of $R(\delta_t)$ if all pairs of points in $F$ are connected by an edge in the Gilbert graph. Observe that e.g. counting the number $N_t^{(k)}$ of $k$-faces is equivalent to a particular subgraph counting. By definition this is a local U-statistics given by 
$$
N_t^{(k)}=N_t^{(k)}(W, \delta_t)=
\frac{1}{(k+1)!}\sum_{x_1, \dots, x_{k+1}\in \eta^{k+1}_{t,\neq}}
\1(\| x_i-x_j\| \leq \delta_t,\ \forall 1\leq i,j \leq k+1). 
$$
Central limit theorems and a concentration inequality follow immediately from the results for local U-statistics.
We restrict our statements again to the thermodynamic case where $t \delta_t^d$ tends to a constant as $t \to \infty$. Results for other regimes can be found e.g. in Penrose's book.
Because of the local definiton of the Gilbert graph, $N_t^{(k)}$ is a local U-statistic. Theorem \ref{thm:regimes} with $\nu_t = t \delta_t^d$ can be applied. 
\begin{theorem}
Let $N$ be a standard Gaussian random variable. Then constants $c(W)$ exist such that
\begin{align*}
d_{W}(\tilde N_{t}^{(k)},N)\leq c(W) t^{-1/2},\\
d_{K}(\tilde N_{t}^{(k)},N)\leq c(W) t^{-1/2}.
\end{align*}
for $t\geq1$.
\end{theorem}

A concentration inequality follows immediately from Theorem~\ref{th:localLDI}.
Observe that $\mu_t(\X) = t \theta( [W])$.
\begin{theorem}
Denote by $m$ the median of $N_t$. 
Then
\begin{eqnarray*}
\P( | \Pi_t -m |\geq u) 
\leq 4 t \ell_d(W) \exp{- \frac{1}{4(k+1)^2} \frac {u^{\frac 2k}}{(u+m)^{\frac 1k}}  }
\end{eqnarray*}
for 
$ \frac {u^2}{ u+m} \geq c_d $.
\end{theorem}

Much deeper results concerning the topology of random simplicial complexes are contained in \cite{Decreusefondetal2011, Kahle} and \cite{KahMeck}.  We refer the interested reader to the recent survey article by Kahle \cite{Kahlesurvey}

\subsection{Sylvester's constant}

Again we assume that the Poisson point process $\eta$ has an intensity-measure of the form $\mu_t(\cdot)= t \ell_d(\cdot \cap W)$, where $\ell_d$ is Lebesgue measure and $W\subset\R^d$ a compact convex set with $\ell_d(W)=1$. 

As a last example of a U-statistic we consider the following functional related to Sylvester's problem. Originally raised with $k=4$ in 1864, Sylvester's original problem asks for the distribution of the number of vertices of the convex hull of four random points. Put
$$N_t = N_t(W,k)=\sum_{(x_1,\hdots,x_k)\in\eta^k_{\neq}} \1 (x_1,\hdots,x_k \text{ are vertices of conv}(x_1,\hdots,x_k)),$$
which counts the number of $k$-tuples of the process such that every point is a vertex of the convex hull, i.e., the number of $k$-tuples in convex position. 

The expected value of $U$ is then given by
\begin{eqnarray*}
\E N_t
&=&
t^k \P (X_1,\hdots,X_k \text{ are vertices of conv}(X_1,\hdots,X_k))
= t^k  p(W,k),
\end{eqnarray*}
where $X_1,\hdots,X_k$ are independent random points chosen according to the uniform distribution on $W$.  

The question to determine the probability $p(W,k)$ that $k$ random points in a convex set $W$ are in convex position has a long history, see e.g. the more recent development by B{\'a}r{\'a}ny \cite{Bar5}. 
In our setting, the function $t^{-k} N_t$ is an estimator for the probability $p(W,k)$ and we are interested in its distributional properties.

The asymptotic behaviour of $\var (N_t) $ is of order $t^{2k-1}$.
Together with Theorem \ref{thm:geometric-Ustats}, we immediately get the following result showing that the estimator $H$ is asymptotically Gaussian:

\begin{theorem}
Let $N$ be a standard Gaussian random variable. Then there exists a constant $ c(W,k)$ such that
$$d_W\left(\tilde N_t,N\right)\leq  c (W,k) t^{-\frac 1 2}  . $$  
\end{theorem}

\bibliographystyle{plain}

\bibliography{Ustats.bib}

\end{document}